\documentclass[10pt]{article}

\usepackage{amsmath,amsfonts,amssymb,amsthm}
\usepackage{color}
\usepackage{colortbl}
\usepackage{array}
\usepackage{mathrsfs}
\usepackage{dcolumn}
\usepackage{longtable}
\usepackage{hhline}
\usepackage{bm}

\newcommand{\bracketed}[2]{\genfrac{[}{]}{0pt}{0}{#1}{#2}}
\newcommand{\braced}[2]{\genfrac{\{}{\}}{0pt}{0}{#1}{#2}}

\newcommand{\flrfunc}[2]{\genfrac{\lfloor }{\rfloor}{0pt}{0}{#1}{#2}}

\newtheorem{thm}{Theorem}[section]

\newtheorem{rmk}[thm]{Remark}

\title{{\bf Analogies of the Qi formula for some Dowling type numbers \footnote{This research is funded by CNU Center of Research and Development}}}
\author{{\large\bf Roberto B. Corcino$\dag$} \\ {\large\bf Jeneveb T. Malusay} \\ {\large\bf Joy Antonette Cillar} \\ {\large\bf Gladys Jane Rama}\\ {\large\bf Oscar Vincent Silang} \\ {\large\bf Ianna Marie Tacoloy} \\ Department of Mathematics\\ Cebu Normal University\\ Cebu City, Philippines
\\ $^\dag$Department of Science and Technology - \\National Research Council of the Philippines \\Bicutan, Taguig City, Philippines
}

\date{}

\begin{document}

\maketitle
\thispagestyle{empty}

\begin{abstract}
In the paper, the authors establish explicit formulas for the Dowling numbers and their generalizations in terms of generalizations of the Lah numbers and the Stirling numbers of the
second kind. These results generalize the Qi formula for the Bell numbers in terms of the Lah numbers and the Stirling numbers of the second kind.

\bigskip
\noindent {\bf Keywords and phrases.} Bell numbers, Stirling numbers, Lah numbers, Whitney-Lah numbers, Whitney numbers, Dowling numbers, generating function, recurrence relation

\smallskip
\noindent {\bf AMS Subject Classification}. 05A10, 05A19, 11B73
\end{abstract}

\section{\textbf{Introduction}}

Lah numbers, denoted by $L_{n,k}$, were first introduced by Ivo Lah as coefficients of the following relations
\begin{eqnarray}
(-x)_n&=&\sum\limits_{k=0}^{n}L_{n,k}(x)_k, \label{Ldef1}\\ 
(x)_n&=&\sum\limits_{k=0}^{n}L_{n,k}(-x)_k,\label{Ldef2}
\end{eqnarray}
where $(x)_n$ is the well-known Pochammer symbol for the falling factorial defined by
\begin{equation}
(x)_n=
\begin{cases}
x(x-1)\dotsb(x-n+1), & n\ge 1\\
1, & n=0.
\end{cases}
\end{equation}  \label{fall}
Note that
$$(-x)_n=(-1)^n(x)(x+1)\ldots (x+n-1)=(-1)^nx^{(n)}$$
where $x^{(n)}$ is the well-known Pochammer symbol for the rising factorial defined by
\begin{equation}
x^{(n)}=
\begin{cases}
x(x+1)\dotsb(x+n-1), & n\ge 1\\
1, & n=0.
\end{cases}
\end{equation}  \label{rise}
Hence, relation \eqref{Ldef1} can be expressed as
$$x^{(n)}=\sum\limits_{k=0}^{n}(-1)^nL_{n,k}(x)_k=\sum\limits_{k=0}^{n}L(n,k)(x)_k,$$
where the numbers $L(n,k):=(-1)^nL_{n,k}$ are called the signless Lah numbers, which can be interpreted combinatorially as the number of ways to partition an $n$-set into $k$ nonempty linearly ordered subsets \cite{Guo}. If we remove the condition that elements in each subset in the partition are linearly oredered, the number of resulting partitions is equal to the Stirling number of the second kind, denoted by $S(n,k)$. Moreover, the number of partitions of an $n$-set into any number of nonempty subsets gives the Bell numbers 
$$B_n=\sum_{k=0}^nS(n,k).$$

\smallskip
The Lah numbers $L_{n,k}$ satisfy the following recurrence relations
\begin{eqnarray}
L_{n+1,k}&=&-L_{n,k-1}-(n+k)L_{n,k},\\ \label{triLah}
L_{n+1,k}&=&\sum\limits_{i=0}^{n-k+1}(-1)^{i+1}(k+n)_iL_{n-i,k-1},\\ \label{verLah}
L_{n,k}&=&\sum\limits_{i=0}^{n-k}(-1)^{i+1}\langle n+k+1\rangle_iL_{n+1,k+i+1},\label{horiLah}
\end{eqnarray}
with $L_{n,0}=0, L_{n,k}=0$ for $n<k$, and $L_{1,1}=-1$. These relations can be used to compute quickly the first values of $L_{n,k}$.

From earlier studies, the exponential generating function and explicit formula for Lah numbers are respectively given by
\begin{equation}
L_k(t)=\sum\limits_{n\ge 0}L_{n,k}\frac{t^n}{n!}=\frac{1}{k!}\bigg(\frac{-t}{1+t}\bigg)^k,\\\label{LGF}
\end{equation}
and
\begin{equation}
L_{n,k}=(-1)^n\binom{n-1}{k-1}\frac{n!}{k!}.\label{LEF}
\end{equation}
Hence, the signless Lah numbers can be expressed as
\begin{equation*}
L(n,k)=\binom{n-1}{k-1}\frac{n!}{k!}.
\end{equation*}
These are exactly the numbers that appeared in the expansion of the $nth$ derivative of the exponential function $e^{\pm 1/t}$. That is,
$$\frac{{\rm d}^n}{{\rm d}t^n}e^{\pm 1/t}=(-1)^ne^{\pm 1/t}\sum_{k=0}^n(\pm 1)^kL(n,k)\frac{1}{t^{n+k}},$$
(see \cite{Boyadzhiev, Qi1, Zhang}). By taking $t=e^{x}$ and using the famous Faa di Bruno's Formula given by
$$\frac{{\rm d}^n}{{\rm d}t^n}f\circ g(t)=\sum_{k=0}^nf^{(k)}(h(t))B_{n,k}(h'(t), h''(t),\ldots h^{(n-k+1)}(t))$$
where $B_{n,k}(x_1, x_2,\ldots, x_{n-k+1})$ is the Bell polynomial of the second kind satisfying
$$B_{n,k}(x_1, x_2,\ldots, x_{n-k+1})=\sum_{{1\leq i\leq n, l_i\in\mathbb{N}\atop \sum_{i=0}^nil_i=n}\atop \sum_{i=0}^nl_i=k }\frac{n!}{\prod_{i=1}^{n-k+1}l_i!}\prod_{i=1}^{n-k+1}\left(\frac{x_i}{i!}\right)^{l_i}$$
with $B_{n,k}(1, 1,\ldots, 1)=S(n,k)$, F. Qi \cite{Qi} proved that the Bell numbers can be expressed in terms of the Lah numbers and the Stirling numbers of the second kinds by
\begin{equation}\label{expBell}
B_n=\sum_{k=0}^n(-1)^{n-k}\left\{\sum_{j=0}^kL(k,j)\right\}S(n,k).
\end{equation}
We call \eqref{expBell} the Qi formula for the Bell numbers. In the same paper \cite{Qi}, F. Qi provided an alternative proof for \eqref{expBell} using the identity for the Lah numbers expressed as the sum of the product of the Stirling numbers of the first and second kinds given by
\begin{equation}\label{ordlahstirling}
L_{n,k}=\sum\limits_{j=k}^{n}(-1)^js(n,j)S(j,k)
\end{equation}
and the inverse relation between the Stirling numbers of the first and second kinds
$$f_n=\sum_{k=0}^nS(n,k)g_k \Longleftrightarrow g_n=\sum_{k=0}^ns(n,k)f_k.$$ 
Other properties of the Bell numbers and that of the Lah and the Stirling numbers can be found in the recent papers of F. Qi \cite{Qi1, Qi2, Qi3}.

\smallskip
In this paper, we express some Dowling type numbers (also known as the Bell type numbers) in terms of certain generalizations of the Lah numbers and the Stirling numbers of the second kind.

\section{\textbf{Whitney-Lah Numbers}}

Whitney numbers of Dowling lattices have been introduced by Dowling \cite{Dowling} based on a class of geometric lattices associated with a finite group $G$ of order $m$. These numbers satisfy the following relations
\begin{eqnarray*}
m^n\left(\frac{x-1}{m}\right)_n&=& \sum_{k=0}^{n}w_m(n,k)x^k,\\
		 x^n&=& \sum_{k=0}^{n}m^kW_m(n,k)\left(\frac{x-1}{m}\right)_k,
\end{eqnarray*}
which can be written as
\begin{eqnarray}
(x-1|m)_n&=& \sum_{k=0}^{n}w_m(n,k)x^k,\label{w11}\\
		 x^n&=& \sum_{k=0}^{n}W_m(n,k)(x-1|m)_k,\label{w22}
\end{eqnarray}
where $m$ is a positive integer and
\begin{equation}
(x|m)_k =\prod_{i=0}^{k-1}(x-im)
\end{equation}
is called the {\it generalized factorial}. The Whitney numbers of the second kind satisfy the following recurrence relation
\begin{equation}\label{recur1}
W_m(n,k)=W_m(n-1,k-1)+(1+mk)W_m(n-1,k).
\end{equation}
Further investigation has been done by Benoumhani \cite{Benoumhani} that yields several properties for $W_m(n,k)$, including an interesting relation 
\begin{equation*}
W_m(n,k)=\sum\limits_{i=k}^{n}\binom{n}{i}m^{i-k}S(i,k), \quad n\ge 0, 0\le k\le n.
\end{equation*}
It is known that Benoumhani was the first one to call the numbers 
$$D_n(m)=\sum_{k=0}^nW_m(n,k)$$
as the {\it Dowling numbers}. Parallel to the Qi formula \eqref{expBell}, it is interesting to derive an explicit formula for the Dowling numbers in terms of certain generaizations of the Lah numbers (Whitney-Lah numbers) and the Whitney numbers of the second kind.

\smallskip
Now, parallel to (\ref{Ldef1}) and (\ref{Ldef2}), we can define the Whitney-Lah numbers, denoted by $L^W_{n,k}(\alpha)$, as coefficients of the following relations
\begin{equation}\label{GLahdef1}
(-x-1|\alpha)_n=\sum_{k=0}^{n}L^W_{n,k}(\alpha)(x-1|\alpha)_k
\end{equation}
and
\begin{equation}\label{GLahdef2}
(x-1|\alpha)_n=\sum_{k=0}^{n}L^W_{n,k}(\alpha)(-x-1|\alpha)_k.
\end{equation}

\smallskip
One important property of these Whitney-Lah numbers is their inverse relation. To obtain this, first we need to establish the following orthogonality relation. 

\smallskip 
\begin{thm} The Whitney-Lah numbers satisfy the following orthogonality relations
\begin{equation}\label{ortho}
\sum_{k=p}^{n}L^W_{n,k}(\alpha)L^W_{k,p}(\alpha)=\delta_{n,p}
\end{equation}
where $\delta_{n,p}$ is the Kronecker delta.
\end{thm}
\begin{proof}
Using the relations in \eqref{GLahdef1} and \eqref{GLahdef2}, we get
\begin{align*}
(-x-1|\alpha)_n&= \sum_{k=0}^{n}L^W_{n,k}(\alpha)(x-1|\alpha)_k= \sum_{k=0}^{n}L^W_{n,k}(\alpha)\sum_{p=0}^{k}L^W_{k,p}(\alpha)(-x-1|\alpha)_p\\
			 &= \sum_{p=0}^{n}\left\lbrace \sum_{k=p}^{n}L^W_{n,k}(\alpha)L^W_{k,p}(\alpha)\right\rbrace (-x-1|\alpha)_p
\end{align*}
Thus, we obtain the desired orthogonality relation.
\end{proof}

\smallskip
The following theorem contains an inverse relation for the Whitney-Lah numbers, which can be shown using the above orthogonality relation.

\smallskip
\begin{thm}
The Whitney-Lah numbers satisfy the following inverse relations:
\begin{align}
f_n&= \sum_{k=0}^{n}L^W_{n,k}(\alpha)g_k\Longleftrightarrow g_n=\sum_{k=0}^{n}L^W_{n,k}(\alpha)f_k\label{inv1}\\
f_k&= \sum_{n=k}^{\infty}L^W_{n,k}(\alpha)g_n\Longleftrightarrow g_k=\sum_{n=k}^{\infty}L^W_{n,k}(\alpha)f_n\label{inv2}
\end{align}
\end{thm}

\smallskip
To generate quickly the first values of the Whitney-Lah numbers, we need to establish some recurrence relations. Here, we give three types of recurrence relations: the triangular recurrence relation, vertical recurrence relation and horizontal recurrence relation.

\smallskip
\begin{thm}\label{triGLahthm}
The Whitney-Lah numbers satisfy the following triangular, vertical and horizontal recurrence relations:
\begin{align}
L^W_{n+1,k}(\alpha)&=-L^W_{n,k-1}(\alpha)-((k+n)\alpha+2) L^W_{n,k}(\alpha),\label{triGLah}\\
L^W_{n+1,k}(\alpha)&=\sum\limits_{i=0}^{n-k+1}(-1)^{i+1}\{n+k\}^{\underline{i}}_{\alpha}L^W_{n-i,k-1}(\alpha),\label{verGLah}\\
L^W_{n,k}(\alpha)&=\sum\limits_{i=0}^{n-k}(-1)^{i+1}\{n+k+1\}^{\overline{i}}_\alpha L^W_{n+1,k+i+1}(\alpha)\label{horiGLah}
\end{align}
where $\{x\}^{\underline{i}}_{\alpha}:=\prod_{j=0}^{i-1}((x-i)\alpha+2)$ and $\{x\}^{\overline{i}}_{\alpha}:=\prod_{j=0}^{i-1}((x+j)\alpha+2)$.
\end{thm}

\begin{proof}
To prove \eqref{triGLah}, we write \eqref{GLahdef1} as 
$$\sum_{k=0}^{n+1}L^W_{n+1,k}(\alpha)(x-1|\alpha)_k=(-x-1|\alpha)_n(-x-1-n\alpha)\qquad\qquad\qquad\qquad\qquad\qquad$$
\begin{align*}
	&\;\;\;\;\;=(-1)\bigg[\sum_{k=0}^{n}L^W_{n,k}(\alpha)(x-1|\alpha)_k(x-1-k\alpha+k\alpha+n\alpha+2)\bigg]\\
	&\;\;\;\;\;=-\sum_{k=0}^{n+1}L^W_{n,k-1}(\alpha)(x|\alpha)_k-\sum_{k=0}^{n+1}((k+n)\alpha+2) L^W_{n,k}(\alpha)(x-1|\alpha)_k\\
	&\;\;\;\;\;=\sum_{k=0}^{n+1}\big[-L^W_{n,k-1}(\alpha)-((k+n)\alpha+2) L^W_{n,k}(\alpha)\big](x-1|\alpha)_k.
\end{align*}
Then, comparing coefficients of $(x-1|\alpha)_k$ yields the desired triangular recurrence relation.

The vertical recurrence relation in \eqref{verGLah} can be obtained by repeated application of \eqref{triGLah}. That is, we have
$$L^W_{n+1,k}(\alpha)=(-1)L^W_{n,k-1}(\alpha)+(-1)^2((k+n)\alpha+2) L^W_{n-1,k-1}(\alpha)\qquad\qquad\qquad\qquad\qquad\qquad$$
\begin{align*}		
		&\;\;\;\;+(-1)^3((k+n)\alpha+2)((k+n-1)\alpha+2)L^W_{n-2,k-1}(\alpha)\\
		&\;\;\;\;+(-1)^4\alpha^3\prod_{i=0}^{2}((k+n-i)\alpha+2)L^W_{n-3,k-1}(\alpha)\\
		&\;\;\;\;+\dotsb+(-1)^{n-k+2}\alpha^{n-k+1}\prod_{i=0}^{n-k}((k+n-i)\alpha+2)L^W_{n-(n-k+1),k-1}(\alpha).
\end{align*}
which is exactly (\ref{verGLah}). To prove the horizontal recurrence relation in \eqref{horiGLah}, we simply evaluate the right-hand side of \eqref{horiGLah} using \eqref{triGLah}.
\end{proof}

Using recurrence relation in \eqref{triGLah}, we can generate quickly the following table of values for $L^W_{n,k}(\alpha)$.

{
\begin{center}
\begin{tabular}{c|llllll}
$n\backslash k$ & 0 & 1 & 2 & 3\\
\hline
0 & 1\\
1 & $-2$ &$-1$\\
2 & $2(\alpha+2)$ & $2(\alpha+2)$ & 1\\
3 & $-4(\alpha+1)(\alpha+2)$ & $-6(\alpha+1)(\alpha+2)$ &$-6(\alpha+1)$ & $-1$
\end{tabular}

\smallskip
{\it Table 1}
\end{center}
}

\begin{rmk}\rm
Recently, F. Qi \cite{Qi5, Qi6} derived interesting forms of recurrence relations, called {\it diagonal recurrence relations}, for the Stirling numbers of the first and second kinds. The method used in deriving such recurrence relation is quite impressive. Hence, it is also interesting to derive diagonal recurrence relations for the Whitney-Lah numbers using this method. 
\end{rmk}

\section{\textbf{Explicit formula for Dowling numbers}}

One can easily verify using (\ref{w11}) and (\ref{w22}) that the Whitney numbers satisfy the following orthogonality and inverse relations
\begin{align}
&\;\sum_{k=j}^{n}W_{\alpha}(n,k)w_{\alpha}(k,j)=\sum_{k=j}^{n}w_{\alpha}(n,k)W_{\alpha}(k,j)=\delta_{n,j}\\
&\;f_n=\sum_{k=0}^{n}w_{\alpha}(n,k)g_k\Longleftrightarrow g_n=\sum_{k=0}^{n}W_{\alpha}(n,k)f_k.\label{in.w}
\end{align}
The following theorem contains a relation which is parallel to the above orthogonality relation of Whitney numbers.

\begin{thm}
The Whitney-Lah numbers satisfy
\begin{equation}
L^W_{n,j}(\alpha)=\sum\limits_{k=j}^{n}(-1)^kw_{\alpha}(n,k)W_{\alpha}(k,j).\label{wla1}
\end{equation}
\end{thm}
\begin{proof}
Note that \eqref{w11} can be rewritten as
\begin{equation}
(-t-1|\alpha)_n=\sum\limits_{k=0}^{n}(-1)^kw_{\alpha}(n,k)t^k.\label{w55}
\end{equation}
Substitute \eqref{w22} in \eqref{w55},
\begin{align*}
\sum\limits_{j=0}^{n}L^W_{n,j}(\alpha)(t-1|\alpha)_j&=\sum\limits_{k=0}^{n}(-1)^kw_{\alpha}(n,k)\sum\limits_{j=0}^{k}W_{\alpha}(k,j)(t-1|\alpha)_j\\
			 &=\sum\limits_{k=0}^{n}\sum\limits_{j=0}^{k}(-1)^kw_{\alpha}(n,k)W_{\alpha}(k,j)(t-1|\alpha)_j.
\end{align*}
Thus, comparing coefficients of $(t-1|\alpha)_j$ proves the theorem.
\end{proof}

The following theorem contains an explicit formula for the Dowling numbers which is expressed in terms of Whitney numbers of the second kind and the Whitney-Lah numbers.

\begin{thm}\label{expDow}
The Dowling numbers equal
\begin{equation}\label{dow1}
D_{n}(\alpha)=\sum_{j=0}^{n}(-1)^{n-j}\left[ \sum_{k=0}^{j}(-1)^jL^W_{j,k}(\alpha)\right] W_{\alpha}(n,j).
\end{equation}
\end{thm}
\begin{proof}
The identity in (\ref{wla1}) can be expressed as
\begin{equation}
L^W_{n,k}(\alpha)= \sum_{j=k}^{n}w_{\alpha}(n,j)(-1)^jW_{\alpha}(j,k).\label{lw}
\end{equation}
Using the inverse relation in \eqref{in.w} with 
$$g_j=(-1)^jW_{\alpha}(j,k)\;\; \mbox{and} \;\;f_n=L^W_{n,k}(\alpha),$$ 
equation \eqref{lw} yields
\begin{align*}
  W_{\alpha}(n,k)&= \sum_{j=0}^{n}(-1)^nW_{\alpha}(n,j)L^W_{j,k}(\alpha)\\
  D_{n}(\alpha) &= \sum_{j=0}^{n}\sum_{k=0}^{n}(-1)^nW_{\alpha}(n,j)L^W_{j,k}(\alpha).
\end{align*}
\end{proof}

The following table contains the first values of Dowling numbers $D_n(\alpha)$ and Whitney numbers $W(n,k;\alpha)$ when $\alpha=3$.

\bigskip
\begin{center}
$\;\;\;\;W_3(n,k)$\\
\begin{tabular}{l|c|lllllll}
$D_n(3)$ & $n\backslash k$ & 0 & 1 & 2 & 3 \\
\hline
1 & 0 & 1\\
2 & 1 & 1 &1\\
7 & 2 & 1 & 5 & 1\\
35 & 3 & 1 & 21 & 12 & 1
\end{tabular}

\smallskip
{\it Table 2}
\end{center}
Now, we can verify the formula in \eqref{dow1} using Tables 1 and 2. For $n=3$ and $\alpha=3$, we have
\begin{align*}
D_3(3)&=(-1)^3\sum\limits_{j=0}^{3}(-1)^{3-j}\left[\sum\limits_{k=0}^{3}(-1)^jL^W_{j,k}(3)\right]W_3(3,j)\\
	  &=-(1(1)-3(21)+21(12)-225(1))=35.
\end{align*}
Note that these values of $D_2(3)$ and $D_3(3)$ coincide with the values that appeared in the first column of Table 2.

\begin{rmk}\rm
When $\alpha=1$, (\ref{dow1}) reduces to 
\begin{equation}\label{specialcase}
D_{n}(1)=\sum_{j=0}^{n}(-1)^{n-j}\left[ \sum_{k=0}^{n}(-1)^jL^W_{j,k}(1)\right] W_1(n,j).
\end{equation}
It is worth noting that, when $\alpha=1$, the above-mentioned finite group $G$ is just a trivial group and that the Dowling lattice is isomorphic to $\Pi_{n+1}$, the lattice of partitions of an $(n + 1)$-element set. It follows that
$$B_{n+1}=D_n(1) \;\;\mbox{and}\;\; S(n+1, j+1) = W_1(n,j).$$
Hence, \eqref{specialcase} can be expressed as
$$B_{n+1}=\sum_{j=0}^{n}(-1)^{n-j}\left[ \sum_{k=0}^{n}(-1)^jL^W_{j,k}(1)\right] S(n+1, j+1).$$
We can verify this formula for particular values of $n$, say $n=3$. That is, 
\begin{align*}
B_{4}&=\sum_{j=0}^{3}(-1)^{3-j}\left[ \sum_{k=0}^{3}(-1)^jL^W_{j,k}(1)\right] S(4, j+1)\\
&=(-1)[1]+(1)[3](7)+(-1)[13](6)+(1)[73](1)=15.
\end{align*}
\end{rmk}

\section{{\bf Explicit formula for $r$-Bell numbers}}
In 1984, certain generalizations of the classical Stirling numbers were introduced by Broder \cite{BRODER}. These numbers are called the $r$-Stirling numbers of the first and second kinds which are denoted by $\bracketed{n}{k}_r$ and $\braced{n}{k}_{r}$, respectively. They are defined combinatorially as follows:

\bigskip
\indent \indent $\bracketed{n}{k}_{r}$ := the number of permutations of an $n$-set into $k$ nonempty \\ \indent \indent\indent\indent \indent  cycles such that the numbers $1,2, \ldots , r$ are in distinct \\ \indent \indent\indent\indent \indent cycles.\\
\indent \indent $\braced{n}{k}_{r}$ := the number of partitions of an $n$-set into $k$ nonempty \\ \indent \indent\indent\indent \indent  subsets such that the numbers $1,2, \ldots , r$ are in distinct \\ \indent \indent\indent\indent \indent subsets.

\bigskip
\noindent These numbers have some properties (see \cite{BRODER}), which are parallel to those of the classical Stirling numbers. F. Qi has mentioned in \cite{Qi7} that these numbers are equivalent to the weighted Stirling numbers of Carlitz \cite{CARL2}, denoted by $S_r(n,k)$, which can be generated by 
\[\frac{(e^t-1)^k}{k!}e^{rt}=\sum_{n=0}^{\infty}S_r(n,k)\frac{t^n}{n!}.\]
This is exactly the generating function that generates the $r$-Stirling numbers of the second kind in \cite[p. 250, Theorem 16]{BRODER}.

\bigskip
Recently, Nyul and Racz \cite{NYUL} also defined combinatorially the $r$-Lah numbers parallel to the definition of $r$-Stirling numbers of the second kind. Precisely, the $r$-Lah numbers are defined by

\bigskip
\indent \indent $\flrfunc{n}{k}_r$:= the number of partitions of a set with $n+r$ elements into \\ \indent \indent\indent\indent  $k+r$ nonempty ordered subsets such that $r$ distinguished \\ \indent \indent\indent\indent elements have to be in distinct ordered blocks.

\bigskip
\noindent These numbers satisfy the following identities and relations
\begin{equation}\label{Lah1}
\flrfunc{n}{k}_{r} = \sum_{j=k}^{n}\bracketed{n}{j}_{r}\braced{j}{k}_{r}
\end{equation}

\begin{equation}\label{Lah4}
b_{n} = \sum_{j=0}^{n} \bracketed{n}{j}_r a_j \Longleftrightarrow a_n = \sum_{j=0}^{n} (-1)^{n-j} \braced{n}{j}_{r} b_{j}
\end{equation}
(see \cite{NYUL}). It is worth mentioning that the first values of $r$-Lah numbers can be computed quickly using the triangular recurrence relation \cite{NYUL}
\begin{equation}\label{tr}
 \flrfunc{n+1}{k}_r= \flrfunc{n}{k-1}_r+(n+k+2r) \flrfunc{n}{k}_r.
\end{equation}
The following table can be generated using \eqref{tr} with $r=2$ and $n=1,2,\ldots,5$

\begin{center}
\begin{tabular}{c|lllllllll}
$n/k$ & 0 & 1 & 2 & 3 & 4 & 5\\
\hline
0 &    1\\
1 &    4 &   1\\
2 &   20 &  10 &   1\\
3 &  120 &  90 &  18 &   1\\
4 &  840 & 840 & 252 &  28 &  1\\
5 & 6720 & 8400& 3360& 560 & 40 & 1\\
\end{tabular}

\smallskip
{\it Table 3}
\end{center}

More interesting results for $r$-Lah numbers have been established by Nyul and Racz in \cite{NYUL}, including the unimodality, maximizing indices, exponential generating function, orthogonality relation and some relations that connect $r$-Lah numbers and $r$-Stirling numbers. However, a relation parallel to that in Theorem \ref{expDow} has not been considered yet in \cite{NYUL}.  

\smallskip
In 2011, Mez\H{o} \cite{Mezo1} defined $r$-Bell numbers, denoted by $B_{n,r}$, as the sum of $r$-Stirling numbers of the second kind. That is, 
$$B_{n,r} = \sum_{k=0}^{n} \braced{n}{k}_{r}.$$
Then $B_{n,r}$ can be interpreted as the number of ways to partition a set with $n$ elements such that the first $r$ elements are in distincts subsets in each partition. These numbers have some interesting properties and identities, which can be found in \cite{Mezo1}. The following theorem contains an explicit formula for $r$-Bell numbers which is expressed in terms of $r$-Lah numbers and $r$-Stirling numbers of the second kind. 

\bigskip
\begin{thm}\label{rBell}
For $n \in \mathbb{N}$, the $r$-Bell numbers $B_{n,r}$ can be computed in terms of $r$-Lah numbers $\flrfunc{n}{k}_r$ and $r$-Stirling numbers $\braced{n}{k}_{r}$ of the second kind by
\begin{equation}\label{expB}
B_{n,r} = \sum_{k=0}^{n} (-1)^{n-k}\braced{n}{k}_{r}\sum_{j=0}^{k} \flrfunc{k}{j}_r.
\end{equation}
\end{thm}

\begin{proof} Using \eqref{Lah1} with $b_{n}=  \flrfunc{n}{k}_r$ and $ a_{j} =  \braced{j}{k}_{r}$, the inverse relation in \eqref{Lah4} gives
$$ \braced{n}{k}_{r} =  \sum_{j=0}^{n} (-1)^{n-j} \braced{n}{j}_{r} \flrfunc{j}{k}_r.$$
Replacing $j$ with $k$ gives
\begin{equation*}
\braced{n}{j}_{r} =  \sum_{k=0}^{n} (-1)^{n-k} \braced{n}{k}_{r} \flrfunc{k}{j}_r.
\end{equation*}
Summing up over $j$ from $r$ to $n$ yields
\begin{equation*}
\sum_{j=0}^{n} \braced{n}{j}_{r} = \sum_{j=0}^{n} \sum_{k=0}^{n} (-1)^{n-k} \braced{n}{k}_{r} \flrfunc{k}{j}_r.
\end{equation*}
This is exactly the desired explicit formula.
\end{proof}

\bigskip
To verify formula \eqref{expB}, we will use the following table of values for $r$-Stirling numbers of the second kind and $r$-Bell numbers with $r=2$, 

\begin{center}
\begin{tabular}{c|l|llllllll}
$B_{n,r}$ & $n/k$ & 0 & 1 & 2 & 3 & 4 & 5 & 6\\
\hline
1   & 0 &    1\\
3   & 1 &    2 &   1\\
10  & 2 &    4 &   5 &   1\\
37  & 3 &    8 &  19 &   9 &   1\\
151 & 4 &   16 &  65 &  55 &  14 &  1\\
674 & 5 &   32 & 211 & 285 & 125 & 20 & 1\\
\end{tabular}

\smallskip
{\it Table 4}
\end{center}

This can be generated using the following triangular recurrence relation
$$\braced{n+1}{k}_{r}=\braced{n}{k-1}_{r}+(k+r)\braced{n}{k}_{r}.$$
Table 4 shows that $B_{3,2}=37$. Using formula \eqref{expB}, we have
\begin{align*}
B_{3,2}&=(-1)^3(1)(8)+(-1)^2(4+1)(19)+(-1)^1(20+10+1)(9)\\
&\;\;\;\;+(-1)^0(120+90+18+1)(1)=37.
\end{align*}
This confirms the value of $B_{3,2}$.

\bigskip
\section{{\bf Explicit formula for $r$-Dowling numbers}}
The $r$-Whitney numbers of the second kind are defined in \cite{Mezo} by
\begin{equation*}
m^n(x)_n=\sum_{k=0}^n(-1)^{n-k}w_{m,r}(n,k)(mx+r)^k
\end{equation*}
and 
\begin{equation*}
(mx+r)^n=\sum_{k=0}^nm^kW_{m,r}(n,k)(x)_k.
\end{equation*}
These numbers have some interesting properties including recurrence relations, generating functions, explicit formulas and their connection with Bernoulli numbers (see \cite{MERCA, Mezo}). Some combinatorial and statistical applications of these numbers are given in \cite{CORCA} which conclude that these numbers are asymptotically normal. Asymptotic formulas are completely discussed in \cite{CORJAY, ACALA, VEGACB1, VEGACB2}. One can easily verify that these numbers satisfy the inverse relation
\begin{equation}\label{inv}
f_n=\sum_{j=0}^{n}(-1)^{n-j}w_{m,r}(n,j)g_j\Longleftrightarrow g_n=\sum_{j=0}^{n}W_{m,r}(n,j)f_j.
\end{equation}
On the other hand, the $r$-Whitney-Lah numbers are defined by Cheon and Jung \cite{CHEON} by
\begin{equation}\label{rwhitneylah}
L_{m,r}(n,k)=\sum_{j=k}^{n}w_{m,r}(n,j)W_{m,r}(j,k),
\end{equation}
which is parallel to \eqref{ordlahstirling}. Several properties of $L_{m,r}(n,k)$ have been derived through factorization of the $r$-Whitney-Lah matrix $[L_{m,r}(n,k)]_{n,k\geq0}$ (see \cite{CHEON}) including the triangular recurrence relation
\begin{equation}\label{triWLah}
L_{m,r}(n,k)=L_{m,r}(n-1,k-1)+(2r + (n + k - 1)m)L_{m,r}(n-1,k),
\end{equation} 
{explicit formula}
\begin{equation}\label{exprWLah}
L_{m,r}(n,k)=\binom{n}{k}\frac{[2r|m]_n}{[2r|m]_k}
\end{equation}
and {horizontal generating function} 
$$[x+2r|m]_n=\sum_{k=0}^nL_{m,r}(n,k)(x|m)_k$$
where
\begin{equation*}
[x|m]_n=(x)(x+m)\ldots (x+(n-1)m)
\end{equation*}
and
\begin{equation*}
(x|m)_k=x(x-m)(x-2m)\ldots (x-(k-1)m).
\end{equation*}

\smallskip
The triangular recurrence relation in \eqref{triWLah} can be used to obtain two more forms of recurrence relations. More precisely, by repeated application of \eqref{triWLah}, we immediately obtain the following vertical recurrence relation for $r$-Whitney-Lah numbers:
\begin{equation}\label{vert_recur}
L_{m,r}(n+1,k+1)=\sum_{j=k}^n(2r + (n + k + 1)m|m)_{n-j}L_{m,r}(j,k).
\end{equation}
Moreover, it can be easily verified using \eqref{triWLah} that $L_{m,r}(n,k)$ satisfy the following horizontal recurrence relation
\begin{equation}\label{hori_recur}
L_{m,r}(n,k)=\sum\limits_{i=0}^{n-k}(-1)^{i}[2r+(n+k+1)m|m]_iL_{m,r}(n+1,k+i+1).
\end{equation}
Note that, by taking $m=1$, \eqref{vert_recur} and \eqref{hori_recur} respectively give the following recurrence relations 
\begin{equation*}
\flrfunc{n+1}{k+1}_r= \sum_{j=k}^n(n+k+2r+1)_{n-j} \flrfunc{j}{k}_r
\end{equation*}
and
\begin{equation*}
\flrfunc{n}{k}_r=\sum\limits_{i=0}^{n-k}(-1)^{i}\langle n+k+2r+1\rangle_i\flrfunc{n+1}{k+i+1}_r.
\end{equation*}

\smallskip
On the other hand, using the explicit formula of $L_{m,r}(n,k)$, one can easily verify that 
$$L_{m,r}(n,k-1)+L_{m,r}(n,k+1)<[L_{m,r}(n,k)]^2$$
is equivalent to $k(n-k)(2r+(k-1)m)<(k+1)(n-k+1)(2r+km)$, which is obviously true. This implies that $(L_{m,r}(n,k))_{k=0}^n$ is strictly log-concave and is, consequently, unimodal.

\smallskip
To derive an analogue version of the explicit formula in Theorem \ref{rBell}, we rewrite first the relation in \eqref{rwhitneylah} as follows
\begin{equation*}
(-1)^nL_{m,r}(n,k)=\sum_{j=k}^{n}(-1)^{n-j}w_{m,r}(n,j)(-1)^jW_{m,r}(j,k).
\end{equation*}
Now, with $g_j=(-1)^jW_{m,r}(j,k)$ and $f_n=(-1)^nL_{m,r}(n,k)$, the inverse relation in (\ref{inv}) yields
\begin{equation*}
				W_{m,r}(n,k)=\sum_{j=0}^{n}(-1)^{n-j}W_{m,r}(n,j)L_{m,r}(j,k) 
\end{equation*}
and
\begin{equation*}
				D_{m,r}(n)=\sum_{j=0}^{n}(-1)^{n-j}\left[\sum_{k=0}^{n}L_{m,r}(j,k) \right]W_{m,r}(n,j). 
\end{equation*}
This result is formally stated in the following theorem.

\begin{thm}\label{rWhitneNumber}
The $r$-Dowling numbers equal
\begin{equation}\label{expl_rDow}
D_{m,r}(n)=\sum\limits_{j=0}^{n}(-1)^{n-j}\left[\sum\limits_{k=0}^{j}L_{m,r}(j,k)\right]W_{m,r}(n,j).
\end{equation}
\end{thm}
To verify this formula for a specific value of $n, m, r$, we need the following table of values for $D_{m,r}(n)$ and $W_{m,r}(n,k)$ with $m=2$ and $r=2$:
\begin{center}
$\qquad\qquad W_{2,2}(n,k)$\\
\begin{tabular}{c|c|lllll}
$D_{2,2}(n)$ & $n/k$ & 0 & 1 & 2 & 3 & 4\\
\hline
1 & 0 & 1 & 0 & 0 & 0 & 0\\
3 & 1 & 2 & 1 & 0 & 0 & 0\\
11 & 2 & 4 & 6 & 1 & 0 & 0\\
49 & 3 & 8 & 28 & 12 & 1 & 0\\
257 & 4 & 16 & 120 & 100 & 20 & 1\\
\end{tabular}
\end{center}
and the table of values for $L_{m,r}(n,k)$ where $m=2$ and $r=2$, which can be generated using \eqref{triWLah}:
\begin{center}
\begin{tabular}{c|llllll}
$n/k$ & 0 & 1 & 2 & 3 & 4 & $\sum\limits_{k=0}^{n}L_{2,2}(j,k)$\\
\hline
0 & 1 & 0 & 0 & 0 & 0 & 1\\
1 & 4 & 1 & 0 & 0 & 0 & 5\\
2 & 24 & 12 & 1 & 0 & 0 & 37\\
3 & 192 & 144 & 24 & 1 & 0 & 361\\
4 & 1920 & 1920 & 480 & 40 & 1 & 4361\\
\end{tabular}
\end{center}

\bigskip
\noindent Using the explicit formula in Theorem \ref{rWhitneNumber}, we get
\begin{align*}
D_{2,2}(4)&=\sum\limits_{j=0}^{4}(-1)^{4-j}\left[\sum\limits_{k=0}^{j}L_{2,2}(j,k)\right]W_{2,2}(4,j)\\
		&=(1)(16)-(5)(120)+(37)(100)-(361)(20)+(4361)(1)\\
		&=257.
\end{align*}
This is exactly the value of $D_{2,2}(4)$ that appeared in the table.

\bigskip
\section{Further Generalization}
There are still many forms of generalization of Bell and Stirling numbers that can possibly be given explicit formula analogous to those in Theorems \ref{expDow}, \ref{rBell} and \ref{rWhitneNumber}. For instance, the generalised Stirling numbers of the first and second kind that appeared in the paper by Tauber \cite{Tauber} are defined respectively by
\begin{equation}\label{tauber1}
P_k(x,n)=\sum_{j=0}^nC_{k,n}^jx^j, \;\;k=1,2
\end{equation}
and
\begin{equation}\label{tauber2}
x^n=\sum_{j=0}^nD_{k,n}^jP_k(x,j), \;\;k=1,2
\end{equation}
for any two given sequences of polynomials $P_1(x,n)$ and $P_2(x,n)$. In relation to this, certain generalized Lah numbers $L_{k,h,n}^j$ are defined by Tauber \cite{Tauber} as
\begin{equation}\label{tauber3}
P_k(x,n)=\sum_{j=0}^nL_{k,h,n}^jP_h(x,j),\;\;\; k\neq h.
\end{equation}
Tauber has proved that these numbers satisfy
\begin{equation}\label{rel1}
L_{k,h,n}^m=\sum_{j=m}^nC_{k,n}^jD_{h,j}^m.
\end{equation}
For appropriate choice of polynomial $P_k(x,n)$, different known pair of Stirling-type numbers can be deduced from the pair of numbers $\{C_{k,n}^j, D_{k,n}^j\}$. For instance, 
by taking $P_k(x,n)=((-1)^{k-1}x-1|\alpha)_n$, we have 
\begin{align*}
\{C_{1,n}^j, D_{1,n}^j\}&=\left\{w_{\alpha}(n,j), W_{\alpha}(n,j)\right\},\\
\{C_{2,n}^j, D_{2,n}^j\}&=\left\{(-1)^jw_{\alpha}(n,j),(-1)^nW_{\alpha}(n,j)\right\},\\
L_{2,1,n}^j&=L_{n,j}^W(\alpha),
\end{align*} 
and equation \eqref{rel1} yields
$$L_{n,m}^W(\alpha)=L_{2,1,n}^m=\sum_{j=m}^nC_{2,n}^jD_{1,j}^m=\sum_{j=m}^n(-1)^jw_{\alpha}(n,j)W_{\alpha}(j,m)$$
which is exactly the identity in \eqref{wla1}. If we let $P_k(x,n)=((-1)^{k-1}x-(1-(-1)^{k-1})r)_n$, we have
\begin{align*}
\{C_{1,n}^j, D_{1,n}^j\}&=\left\{(-1)^{n-k}\bracketed{n}{j}_r, \braced{n}{j}_r\right\},\\
\{C_{2,n}^j, D_{2,n}^j\}&=\left\{(-1)^{n}\bracketed{n}{j}_r,(-1)^n\braced{n}{j}_r\right\},\\
L_{2,1,n}^j&=(-1)^n\flrfunc{n}{j}_r,
\end{align*} 
and equation \eqref{rel1} yields
$$(-1)^n\flrfunc{n}{m}_r=L_{2,1,n}^m=\sum_{j=m}^nC_{2,n}^jD_{1,j}^m=\sum_{j=m}^n(-1)^{n}\bracketed{n}{j}_r\braced{j}{m}_r$$
or
$$\flrfunc{n}{m}_r=\sum_{j=m}^n\bracketed{n}{j}_r\braced{j}{m}_r$$
which is exactly the identity in \eqref{Lah1}. Moreover, by taking $P_k(x,n)=(x+(1-(-1)^{k-1}r|(-1)^{k-1}m)_n$, we have 
\begin{align*}
\{C_{1,n}^j, D_{1,n}^j\}&=\left\{(-1)^{n-j}w_{\alpha,r}(n,j), W_{\alpha,r}(n,j)\right\},\\
\{C_{2,n}^j, D_{2,n}^j\}&=\left\{w_{\alpha,r}(n,j),(-1)^{n-j}W_{\alpha,r}(n,j)\right\},\\
L_{2,1,n}^j&=L_{\alpha,r}(n,j),
\end{align*} 
and equation \eqref{rel1} yields
$$L_{\alpha,r}(n,m)=L_{2,1,n}^m=\sum_{j=m}^nC_{2,n}^jD_{1,j}^m=\sum_{j=m}^nw_{\alpha,r}(n,j)W_{\alpha,r}(j,m)$$
which is exactly the identity in \eqref{rwhitneylah}.

\smallskip
In a separate paper of Tauber \cite{Tauber1}, the above generalised Stirling numbers have been shown to satisfy
$$\sum_{m=j}^nC_{k,n}^mD_{k,m}^j=\sum_{m=j}^nD_{k,n}^jC_{k,m}^j=\delta_{n,j},$$
which implies immediately the inverse relation
$$f_n=\sum_{m=0}^nC_{k,n}^mg_m\Longleftrightarrow g_n=\sum_{m=0}^nD_{k,n}^mf_m.$$
Applying this inverse relation to \eqref{rel1} yields
\begin{equation*}
D_{h,n}^m=\sum_{j=0}^nD_{k,n}^jL_{k,h,j}^m,
\end{equation*}
and
\begin{equation}\label{rel2}
\sum_{m=0}^nD_{h,n}^m=\sum_{m=0}^n\sum_{j=0}^nD_{k,n}^jL_{k,h,j}^m.
\end{equation}
Thus, if we denote the right-hand side of \eqref{rel2} by $B_{h,n}$, then we have
\begin{equation}\label{rel3}
B_{h,n}=\sum_{j=0}^n\left[\sum_{m=0}^nL_{k,h,j}^m\right]D_{k,n}^j.
\end{equation}
One can easily verify that, with $k=2$ and $h=1$, \eqref{rel3} gives the explicit formulas in Theorems \ref{expDow}, \ref{rBell} and \ref{rWhitneNumber}.

\smallskip
The unified generalization of Stirling numbers $\{S^1(n,k), S^2(n,k)\}$ are defined by Hsu and Shiue \cite{Hsu} as coefficients of the following expansion
\begin{equation}\label{ug1}
(t|\alpha)_n=\sum_{k=0}^nS^1(n,k)(t-\gamma|\beta)_k
\end{equation}
and
\begin{equation}\label{ug2}
(t|\beta)_n=\sum_{k=0}^nS^2(n,k)(t+\gamma|\alpha)_k
\end{equation}
where $\{S^1(n,k), S^2(n,k)\}=\{S(n,k;\alpha,\beta,\gamma), S(n,k;\beta,\alpha,-\gamma)\}$. These numbers satisfy
\begin{equation*}
\sum_{k=n}^mS^1(m,k)S^2(k,n)=\sum_{k=n}^mS^2(m,k)S^1(k,n)=\delta_{m,n}
\end{equation*}
and 
\begin{equation}\label{invrel}
f_n=\sum_{k=0}^nS^1(n,k)g_k \Longleftrightarrow g_n=\sum_{k=0}^nS^2(n,k)f_k.
\end{equation}
To obtain an explicit formula analogous to those in Theorems \ref{expDow}, \ref{rBell} and \ref{rWhitneNumber}, one must define a generalized Lah-type numbers $L(n,j;\alpha,\beta,\gamma)$ such that
\begin{equation}\label{genlah}
L(n,j;\alpha,\beta,\gamma)=\sum_{k=j}^n(-1)^kS^2(n,k)S^1(k,j).
\end{equation}
Then, by applying the inverse relation in \eqref{invrel}, we have
\begin{equation*}
(-1)^{n}S^1(n,j)=\sum_{k=0}^nS^1(n,k)L(k,j;\alpha,\beta,\gamma)
\end{equation*}
and
\begin{equation*}
S(n,j;\alpha,\beta,\gamma)=(-1)^{n}\sum_{k=0}^nS(n,k;\alpha,\beta,\gamma)L(k,j;\alpha,\beta,\gamma).
\end{equation*}
Consequently, the generalized Bell numbers $W_n$ in \cite{Hsu} defined by
$$W_n=\sum_{j=0}^nS(n,j;\alpha,\beta,\gamma),$$
can be expressed in terms of the unified generalization of the Stirling numbers and certain generalized Lah-type numbers as 
\begin{equation}\label{ugexp}
W_n=\sum_{k=0}^n(-1)^n\left[\sum_{j=0}^nL(k,j;\alpha,\beta,\gamma)\right]S(n,k;\alpha,\beta,\gamma).
\end{equation}
It is worth mentioning that the generalized Stirling numbers $c_{n,k}$ defined by N. Caki\'c \cite{Cakic} as
$$(x|\alpha)_n=\sum_{k=0}^nc_{n,k}(x|1)_k,$$
can be expressed in terms of the unified generalization of the Stirling numbers by replacing $\alpha$ with $-\alpha$, $\beta$ with 1 and $\gamma$ with 0. That is,
$$c_{n,k}=S(n,k;-\alpha, 1, 0).$$
Hence, if we denote the sum of $c_{n,k}$ by $\hat{B}_n(\alpha)$, that is,
$$\hat{B}_n(\alpha)=\sum_{k=0}^nc_{n,k},$$
then using \eqref{ugexp} 
$$\hat{B}_n(\alpha)=\sum_{k=0}^n(-1)^n\left[\sum_{j=0}^nL(k,j;-\alpha,1,0)\right]c_{n,k}.$$
Furthermore, with $\alpha=0$ and $\gamma=r$, we get
\begin{align*}
w_{\beta}(n,k)&=S(n,k;\beta,0,-1),\\
W_{\beta}(n,k)&=S(n,k;0,\beta,1),\\
L_{n,k}^W(\beta)&=L(n,k;0,\beta,1),
\end{align*}
and the explicit formula in \eqref{ugexp} coincides with that in Theorem \ref{expDow}. When $\alpha=0$, $\beta=1$ and $\gamma=r$, we have
\begin{equation*}
\bracketed{n}{k}_r=S(n,k;1,0,-r), \;\;\braced{n}{k}=S(n,k;0,1,r), \;\;\flrfunc{n}{k}_r=(-1)^nL(n,k;0,1,r),
\end{equation*}
and the explicit formula in \eqref{ugexp} coincides with that in Theorem \ref{rBell}. Lastly, when $\alpha=0$ and $\gamma=r$, we get
\begin{align*}
w_{\beta,r}(n,k)&=(-1)^{n-k}S(n,k;\beta,0,-r),\\
W_{\beta,r}(n,k)&=S(n,k;0,\beta,r),\\
L_{\beta,r}(n,k)&=(-1)^nL(n,k;0,\beta,r),
\end{align*}
and the explicit formula in \eqref{ugexp} coincides with that in Theorem \ref{rWhitneNumber}.

\smallskip
On the other hand, further extension of the unified generalization of Stirling numbers, denoted by $S(n,k,\alpha,\beta,\gamma;\epsilon)$ is defined by T-X. He \cite{He} as 
$$S(\gamma,\eta,\alpha,\beta,r;\epsilon):=\frac{1}{\beta^{\eta}\Gamma (\eta+1)}\lim_{z\to r}\Delta_{\beta}^{\eta, \epsilon}\left(\langle z\rangle_{\gamma,-\alpha}\right).$$
Several properties for $S(\gamma,\eta,\alpha,\beta,r;\epsilon)$ are established in \cite{He} including triangular recurrence relation, explicit formula and exponential generating function. One can easily verify that the above unified generalization of Stirling numbers by Hsu and Shuie can be deduced from $S(n,k,\alpha,\beta,\gamma;\epsilon)$ using the exponential generating function for $S(\gamma,\eta,\alpha,\beta,r;\epsilon)$ with $\epsilon=0$. That is, 
$$S(n,k;\alpha,\beta,\gamma)=S(n,k,\alpha,\beta,\gamma;0)$$
Then, it would be interesting to consider the following pair and the corresponding Bell-type numbers 
$$\left\{S(n,k,\alpha,\beta,r;\epsilon),S(n,k,\beta,\alpha,-r;\epsilon)\right\}, \;\;W(n;\epsilon)=\sum_{k=0}^nS(n,k,\alpha,\beta,r;\epsilon)$$
as well as Lah-type numbers 
$$L(n,j,\alpha,\beta,\gamma,\epsilon)=\sum_{k=j}^n(-1)^kS(n,k,\beta,\alpha,-r;\epsilon)S(k,j,\alpha,\beta,r;\epsilon)$$
and establish explicit formula analogue to that in \eqref{ugexp}.

\bigskip
\noindent {\bf Acknowledgement}. The authors wish to thank the anonymous referee for reading the manuscript thoroughly which resulted in numerous corrections and improvements.

\end{document}